\documentclass{article}

\usepackage{csquotes} % Fuer enquote-Befehl
\usepackage{ulem} % Fuer uline-Befehl
\usepackage{amsmath}
\usepackage{amsthm}
\usepackage{amsfonts}
\usepackage{amssymb}
\usepackage{graphicx}
\usepackage{paralist} % Fuer compactenum-Umgebung
\usepackage[left=2.5cm,right=2.5cm,top=1.5cm,bottom=1cm,includeheadfoot]{geometry}
\usepackage{hyperref}
\usepackage{fancyhdr}
\usepackage{authblk}

\newtheorem{theorem}{Theorem}[section]
\theoremstyle{break} 

\newtheorem{lemma}[theorem]{Lemma}
\newtheorem{corollary}[theorem]{Corollary}
\newtheorem{definition}[theorem]{Definition}
\newtheorem{proposition}[theorem]{Proposition}
\newtheorem{remark}[theorem]{Remark}

\newtheorem{example}[theorem]{Example}

\newcommand\C{\mathbb C}
\newcommand\D{\mathbb D}
\newcommand\R{\mathbb R}

\newcommand\Z{\mathbb Z}
\newcommand\N{\mathbb N}

\newcommand{\HT}{ \mathfrak{H} } % The Hilbert transformation
\newcommand{\PT}{ \mathfrak{P} } % The Poisson transformation
\newcommand{\HQC}{ \mathcal{H}_{qc}^+ } % Boundary mappings that induce q.c. harmonic automorphisms of the unit disk
\newcommand{\dsup}{d_{\sup}} % Die Supremumsmetrik auf einem Gebiet G in C.
\newcommand\id{\operatorname{id}} % Die Identität auf einer Menge
\newcommand{\Arg}{\operatorname{Arg}} % Principal value of the argument

     \title{Topological Characteristics of Harmonic Quasiconformal Unit Disk Automorphisms in the Uniform Topology}

     \author{Florian Biersack}
     \affil{University of W\"{u}rzburg, Chair for Complex Analysis, Emil--Fischer--Strasse 40, 97074 W\"{u}rzburg, Bavaria, Germany \\ Email address: florian.biersack@mathematik.uni-wuerzburg.de}
     \date{}

\begin{document}
     \maketitle
     \begin{abstract}
     \noindent We study the class $HQ(\mathbb{D})$, the set of harmonic quasiconformal automorphisms of the unit disk $\D$ in the complex plane, endowed with the topology of uniform convergence. Several important topological properties of this space of mappings are investigated, such as separability, compactness, path--connectedness and completeness.
     \end{abstract}

\section{Introduction}

The idea for investigating the harmonic quasiconformal automorphisms of $\D := \left\{ z \in \C \; \big{|} \; |z| < 1 \right\}$ was on the one hand inspired from a topic that has drawn much attention in recent years: The harmonic quasiconformal mappings. Initiated by Martio in 1968 (see \cite[p. 238]{Kalaj} and \cite[p. 366]{Pavlovic}), this particular class of homeomorphisms attracted much interest in the recent past, see \cite[Introduction]{BozinMateljevic}, \cite{Kalaj}, \cite{KrzyzNowak}, \cite{Pavlovic}, \cite[Section 10.3]{PavlovicBook} and the references therein, to name only a few. In particular, Kalaj and Pavlovi\'{c} worked intensively in this area and achieved numerous results, among others several characterization statements for harmonic quasiconformal automorphisms of the unit disk (see \cite[Theorem A, p. 239]{Kalaj} and Proposition \ref{PropositionPavlovicAbstractHQC} below). On the other hand, in \cite{BiersackLauf}, the authors studied the quasiconformal automorphism groups of simply connected domains in the complex plane. For this class of domains, the unit disk in $\C$ can be regarded as the reference element, not least due to the classical Riemann Mapping Theorem (RMT) and its quasiconformal counterpart, the Measurable RMT (see e.g. \cite[Mapping Theorem, p. 194]{LehtoVirtanen}). In view of these circumstances and by the Theorem of Rad\'{o}--Kneser--Choquet (see Proposition \ref{PropositionRKC}), a similarly striking result on harmonic mappings, the following discussion will focus on the special case of harmonic quasiconformal automorphisms of the unit disk in $\C$. The topology used in this paper is the uniform topology, induced by the supremum metric
\[
\dsup(f,g) := \sup \limits_{z \in \D} \left| f(z) - g(z) \right|
\]
for (bounded) mappings $f, g: \D \longrightarrow \C$.
%The required statements and results on harmonic mappings in $\C$ are thereby mainly cited from \cite{AxlerEtAl} and \cite{Duren}.

\section{Definition and basic properties}

%As already mentioned in the introduction, this section is concerned with the following particular subset of quasiconformal automorphisms of the unit disk:
In \cite{BiersackLauf}, the authors studied the following space of mappings:
\begin{definition} \label{DefinitionHQD} %[$HQ(\D)$]
Let $G \subsetneq \C$ be a bounded, simply connected domain, then
\[
Q(G) := \left\{ f: G \longrightarrow G \; \big{|} \; f \text{ is a quasiconformal mapping of } G \text{ onto } G \right\}
\]
\end{definition}

Several central topological properties of $Q(G)$ in the topology of uniform convergence induced by $\dsup$ were studied in \cite{BiersackLauf}. In the unit disk $\D$, a particularly specialized subclass of such mappings arises by demanding the additional property of harmonicity, i.e. by considering
\begin{align}
HQ(\D) := \left\{ f \in Q(\D) \, \big{|} \, f \ \operatorname{ is } \ \operatorname{harmonic} \right\}
\end{align}

That is, the mappings in $HQ(\D)$ are the harmonic quasiconformal automorphisms of $\D$. Here and henceforth, a complex--valued mapping $f = u + iv$ defined on a domain is called \textit{harmonic} if both, its real and imaginary parts, are real--valued harmonic mappings, which in turn are defined via the \textit{Laplace equation}
\[
\Delta u = \frac{\partial^2 u}{\partial x^2} + \frac{\partial^2 u}{\partial y^2} = 0,
\]
the differential polynomial $\Delta := \frac{\partial^2}{\partial x^2} + \frac{\partial^2}{\partial y^2}$ being the \textit{Laplace operator}. Harmonic mappings possess numerous important properties, such as the mean--value property and the maximum principle (\cite[p. 12]{Duren}), which are in turn deeply connected with holomorphic functions by well--known results from Complex Analysis. An immediate conclusion to be drawn is $\Sigma(\D) \subseteq HQ(\D)$, where $\Sigma(\D) := \{ f \in Q(\D) \; \big{|} \; f \text{ is conformal} \}$ denotes the subset of conformal automorphisms of $\D$. In particular, it is $\id_\D \in HQ(\D)$ and therefore $HQ(\D) \not = \emptyset$, where $\id_\D$ is the identity on $\D$. An important fact about harmonic mappings in $\C$ and their representation is given by the following result due to Rad\'{o}, Kneser and Choquet (\cite[pp. 33--34]{Duren}, \cite[pp. 154--156]{DurenSchober} and \cite[Theorem 1.1, p. 5]{PavlovicBook}):

\begin{proposition}[Rad\'{o}--Kneser--Choquet] \hspace{0.1cm} \label{PropositionRKC} \newline
Let $G \subsetneq \C$ be a convex Jordan domain and $\gamma: \partial \D \longrightarrow \partial G$ be a \textbf{weak homeomorphism}, i.e. a continuous mapping of $\partial \D$ onto $\partial G$ such that the preimage $\gamma^{-1}(\xi)$ of each $\xi \in \partial G$ is either a point or a closed subarc of $\partial \D$. Then the \textbf{harmonic extension}
\begin{align}
\PT[\gamma](z) := \frac{1}{2 \pi} \int \limits_0^{2 \pi} \frac{1 - r^2}{1 - 2r \cos(t - \phi) + r^2} \gamma(e^{it}) \, \mathrm{d}t, \; z = r e^{i \phi} \in \D,
\label{FormulaPoissonTransformationHarmonicExtension}
\end{align}
defines an injective harmonic mapping of $\D$ onto $G$; moreover, $\PT[\gamma]$ is unique. Conversely, if $G \subsetneq \C$ is a strictly convex Jordan domain and $f: \D \longrightarrow G$ is an injective harmonic mapping, then $f$ has a continuous extension to $\overline{\D}$ which defines a weak homeomorphism of $\partial \D$ onto $\partial G$. Moreover, if $f \in C(\overline{\D})$ is harmonic in $\D$, then $f |_{\D}$ can be written in the form (\ref{FormulaPoissonTransformationHarmonicExtension}).
%\footnote{A set $S \subseteq \C$ is called \textit{strictly convex} if every point on the line segment connecting $x,y \in \overline{S}$ other than the endpoints is contained in the interior of $S$ (\cite[p. 34]{Duren}). For example, a circle is strictly convex (and in particular convex), while a rectangle is not strictly convex (yet convex).}
%if it is convex ($x,y \in S \Longrightarrow (tx + (1-t)y) \in S \, \forall t \in [0,1]$) and
%with $\PT[\gamma] \in C(\overline{\D})$ and $\PT[\gamma] \equiv \gamma$ on $\partial \D$
\end{proposition}

For a (Jordan) domain $G \subseteq \C$, let $\mathcal{H}^\ast(\partial \D, \partial G)$ denote the set of all weak homeomorphisms of $\partial \D$ onto $\partial G$ and in the special case $G = \D$ define $\mathcal{H}^\ast(\partial \D) := \mathcal{H}^\ast(\partial \D, \partial \D)$. Consequently, let $\mathcal{H}^+(\partial \D, \partial G)$ and $\mathcal{H}^+(\partial \D)$ denote the corresponding subsets of all orientation--preserving homeomorphisms, respectively.

\begin{remark} \hspace{0.1cm} \label{RemarkPoissonTransformation} \newline
The harmonic extension $\PT[\gamma]$ defined by (\ref{FormulaPoissonTransformationHarmonicExtension}) is also called the \textbf{Poisson transformation} of $\gamma \in \mathcal{H}^\ast(\partial \D, \partial G)$, and the corresponding integral kernel%(also called Poisson formula) 
\[
\frac{1 - r^2}{1 - 2r \cos(t) + r^2}
\]
is called the \textbf{Poisson kernel}, see \cite[p. 12]{Duren} and \cite[pp. 5--6]{PavlovicBook}.%Moreover, the Poisson transformation is intimately related to one of the most famous questions in the history of mathematics, the \textbf{Dirichlet problem}, whose solution is given explicitly by (\ref{FormulaPoissonTransformationHarmonicExtension}); see \cite[pp. 12--15]{AxlerEtAl} and \cite[Section 1.4]{Duren}.
\end{remark}

From the Rad\'{o}--Kneser--Choquet Theorem \ref{PropositionRKC}, one obtains the following (see also \cite[pp. 337--338]{KrzyzNowak})

\begin{corollary} \label{CorollaryCharacterizationHQD}
\[
HQ(\D) = Q(\D) \; \bigcap \; \left\{ \PT[\gamma] \; \big{|} \; \gamma \in \mathcal{H}^+(\partial \D) \right\}
\]
\end{corollary}

Corollary \ref{CorollaryCharacterizationHQD} also makes sense when recalling that every quasiconformal automorphism of a Jordan domain admits a homeomorphic boundary extension (see \cite[p. 13]{Lehto}). In particular, the induced boundary mapping is injective, hence an element of $\mathcal{H}^+(\partial \D)$. A concrete harmonic automorphism of the unit disk is visualized in

\begin{example} \hspace{0.1cm} \label{ExampleHarmonicUnitDiskAutomorphism} \newline
For $x \in [0, 1]$, consider the piecewise--defined function
\[
\phi(x) = \begin{cases} 2x, & x \in [0, \frac{1}{3}] \\ \frac{2}{3}, & x \in [\frac{1}{3}, \frac{3}{4}] \\ \frac{4}{3}x - \frac{1}{3}, & x \in [\frac{3}{4}, 1] \end{cases}
\]
which is easily seen to map the interval $[0,1]$ continuously, but not injectively onto itself while keeping the endpoints $x = 0$ and $x = 1$ fixed. Transferring $\phi$ to the interval $[0, 2 \pi]$ by conjugating it via the mapping $x \longmapsto t = 2 \pi x$ yields a function $\widetilde \psi \in C([0, 2 \pi])$ with the same properties. Consequently, the mapping
\begin{align}
\gamma(e^{it}) = e^{i \widetilde \varphi(t)}
\label{FormulaBoundaryFunctionHarmonicExample}
\end{align}
for $e^{it} \in \partial \D$ defines a weak homeomorphism of $\partial \D$ onto itself, i.e. $\gamma \in \mathcal{H}^\ast(\partial \D)$. The corresponding harmonic extension provided by Proposition \ref{PropositionRKC} therefore yields a harmonic homeomorphism $\PT[\gamma]$ of $\D$ onto itself. Figure \ref{FigureHarmonicAutomorphismUnitDisk} shows the (approximated) mapping behaviour of this harmonic extension, visualized by concentric circles around the origin, radial rays and an Euclidean grid. However, the mapping $\PT[\gamma]$ is not quasiconformal due to the fact that its boundary function -- which equals $\gamma$ by construction -- is not injective, but this would necessarily follow.
\end{example}

\begin{figure}
\centering
\includegraphics[width=13cm]{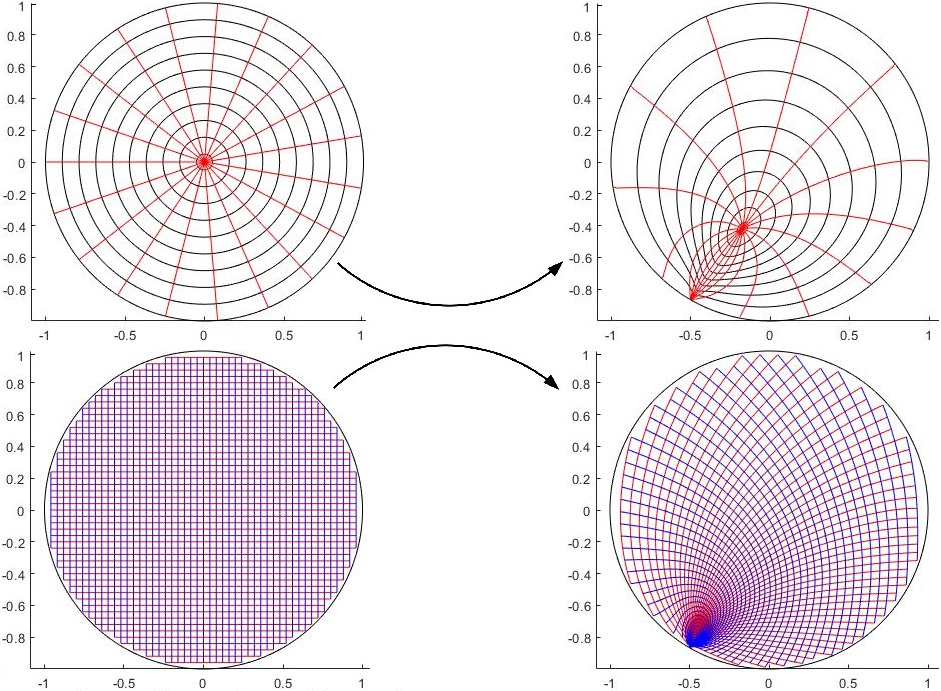}
\caption[An approximation of a certain harmonic unit disk automorphism.]{Preimage (left) and image (right) of concentric circles and radial rays (top) as well as of an Euclidean grid (bottom) in $\D$ under the harmonic extension of $\gamma$ defined by $(\ref{FormulaBoundaryFunctionHarmonicExample})$.}
\label{FigureHarmonicAutomorphismUnitDisk}
\end{figure}

\begin{remark} \hspace{0.1cm} \label{RemarkExampleNonQCHomeomorphism} \newline
In particular, the harmonic extension $\PT[\gamma]$ discussed in Example \ref{ExampleHarmonicUnitDiskAutomorphism}, with $\gamma$ given by (\ref{FormulaBoundaryFunctionHarmonicExample}), provides a concrete example of a sense--preserving homeomorphism of the unit disk that is not quasiconformal, i.e.
\[
\PT[\gamma] \in \mathcal{H}^+(\D) \backslash Q(\D).
\]
Another example of such a mapping will be presented in Proposition \ref{PropositionCounterexampleHQC}.
\end{remark}

%Furthermore, Pavlovi\'{c} showed the following characterization of harmonic quasiconformal automorphisms of $\D$ in terms of their boundary functions (see \cite[Abstract]{Pavlovic} and also Proposition \ref{PropositionCharacterizationHQD} below):
%
%\begin{proposition}[Pavlovi\'{c}] \label{PropositionPavlovicAbstractHQC}\hspace{0.1cm} \newline
%Let $f$ be an orientation--preserving harmonic homeomorphism of the unit disk onto itself. Then the following conditions are mutually equivalent:
%\begin{compactenum}[(i)]
%	\item{$f$ is quasiconformal, i.e. $f \in HQ(\D)$;}
%	\item{$f$ is bi--Lipschitz with respect to the Euclidean metric;}
%	\item{$\varphi$ is bi--Lipschitz and the Hilbert transformation of $\varphi'$ is in $L^\infty(\R)$, where $e^{i \varphi} := f \big{|}_{\partial \D}$ .}
%\end{compactenum}
%\end{proposition}
%
%A mapping $f: X \longrightarrow Y$ between metric spaces $(X, d_X)$ and $(Y, d_Y)$ is called \textit{bi--Lipschitz} if there exists a constant $L \geq 1$ such that
%\[
%\frac{1}{L} d_X(x_1, x_2) \leq d_Y(f(x_1), f(x_2)) \leq L d_X(x_1, x_2)
%\]
%for all $x_1, x_2 \in X$, thus sharpening the classical notion of a Lipschitz--continuous mapping. Furthermore, the Hilbert transformation of . In combination, Corollary \ref{CorollaryCharacterizationHQD} and Proposition \ref{PropositionPavlovicAbstractHQC} give
%\[
%HQ(\D) = \left\{ P[\gamma] \; \big{|} \; \gamma = e^{i \varphi} \in \text{Hom}^+(\partial \D), \varphi \text{ is bi--Lipschitz and } H(\varphi') \in L^\infty(\R) \right\}
%\]

%\subsection{Algebraic aspects of $HQ(\D)$}

A basic fact in the theory of harmonic mappings is that the composition of two such mappings is not necessarily harmonic again (see \cite[p. 2]{Duren}). In the same manner, the inverse mapping of an injective harmonic mapping is also not harmonic in general, except for special situations, as stated in (see \cite[Theorem, pp. 145--148]{Duren})
%\footnote{In the cited statement of Choquet--Deny's Theorem in \cite{Duren}, there is a small error in the formula of the harmonic mapping $f$: In the argument function, Duren writes $\gamma - e^{\beta z}$; however, as becomes clear from the book's proof of this Theorem as well, it should actually be $\gamma - e^{- \beta z}$, i.e. the minus sign in the exponent is missing.}

\begin{proposition}[Choquet--Deny] \hspace{0.1cm} \label{PropositionChoquetDeny} \newline
Suppose $f$ is an orientation--preserving injective harmonic mapping defined on a simply connected domain $G \subseteq \C$, and suppose that $f$ is neither analytic nor affine. Then the inverse mapping $f^{-1}$ is harmonic if and only if $f$ has the form
\begin{align}
\label{FormulaInverseHarmonicMapping}
%f(z) = \frac{\sigma}{\overline{\alpha}}z + \frac{1}{\overline{\alpha}} \log \left( \frac{\mu - e^{- \sigma z}}{\overline{\mu - e^{- \sigma z}}} \right) + \text{const.}
f(z) = \alpha \left( \beta z + 2i \Arg( \gamma - e^{-\beta z} ) \right) + \delta,
\end{align}
%where $\sigma, \alpha, \mu \in \C^\times$ and $|\mu| > \sup \limits_z | e^{- \sigma z}|$.
where $\alpha, \beta, \gamma, \delta \in \C$ are constants with $\alpha \beta \gamma \not = 0$ and $\left| e^{-\beta z} \right| < | \gamma |$ for all $z \in G$.
\end{proposition}

This result and the previously stated facts immediately imply (see also \cite[Problem 10.1, p. 311]{PavlovicBook})

\begin{theorem} \hspace{0.1cm} \newline
$HQ(\D)$ is no semigroup with respect to composition of mappings. In particular, $HQ(\D)$ is no subgroup of $Q(\D)$.
\end{theorem}

%Furthermore, the Choquet--Deny Theorem \ref{PropositionChoquetDeny} yields that $HQ(\D)$ is not closed under inversion. In view of these circumstances, this raises the
%
%\begin{question} \hspace{0.1cm} \newline
%Can a mapping of the form (\ref{FormulaInverseHarmonicMapping}) be an automorphism of the unit disk $\D$ if the parameter values are chosen appropriately? If yes, can such a mapping be quasiconformal? \label{QuestionInverseHarmonicAutomorphismQC}
%\end{question}

\section{Topological properties of $HQ(\D)$}

This section is intended to study some central topological properties of the space $HQ(\D)$. Since this situation is settled in the context of metric spaces, many of these topological notions can be expressed in terms of convergent sequences in the space $HQ(\D)$. Thus, certain convergence results for uniformly convergent sequences of harmonic mappings will prove valuable, as stated in

\begin{proposition}
\label{PropositionConvergenceTheorem} \hspace{0.1cm} \newline
Let $(f_n)_{n \in \N}$ be a sequence of harmonic mappings on a domain $G \subseteq \C$.
\begin{compactenum}[(i)]
	\item{If $(f_n)_n$ converges locally uniformly on $G$ to some function $f$, then $f$ is harmonic (Weierstraß--type Theorem, see \cite[Theorem 1.23, p. 16]{AxlerEtAl}).}%and $f^{(k)}_n$ converges locally uniformly to $f^{(k)}$ for every $k \in \N$.%\footnote{Actually, the cited version of the Weierstraß--type Theorem from \cite{AxlerEtAl} is concerned \enquote{only} with the real and imaginary parts of the harmonic functions $f_n$, but this situation transfers immediately to the general case considered in the current section, since by definition, $f_n$ is harmonic if and only if $\Re(f_n)$ and $\Im(f_n)$ satisfy the Laplace equation.}
	\item{If additionally, all $f_n$ are injective and the sequence $(f_n)_n$ converges locally uniformly on $G$ to $f$, then $f$ is either injective, a constant mapping, or $f(G)$ lies on straight line (Hurwitz--type Theorem, see \cite[Theorem 1.5]{BshoutyHengartner}).}
\end{compactenum}
\end{proposition}

The first result concerning certain topological aspects of $HQ(\D)$ is given in

\begin{theorem} \label{TheoremClosedNotOpen}\hspace{0.1cm} \newline
The space $HQ(\D)$ is separable and non--compact. As a subspace, it is closed in $Q(\D)$.
\end{theorem}
\begin{proof}
If $(f_n)_{n \in \N} \subseteq HQ(\D)$ converges uniformly on $\D$ to $f \in Q(\D)$, then $f$ is harmonic by Proposition \ref{PropositionConvergenceTheorem}(i), hence $f \in HQ(\D)$. Therefore, $HQ(\D)$ is closed in $Q(\D)$. \newline
As for the separability of $HQ(\D)$, it suffices to observe that the ambient metric space $Q(\D)$ is separable by \cite[Theorem 6, p. 5]{BiersackLauf}. The claimed separability of $HQ(\D)$ is then implied by the fact that subspaces of separable metric spaces are also separable. In order to see that $HQ(\D)$ is a non--compact space, suppose the contrary, i.e. $HQ(\D)$ is compact in the uniform topology. Due to the completeness of $\Sigma(\D)$ (see \cite[Satz 1, p. 229]{Gaier}), the space $\Sigma(\D)$ is closed in the ambient space $HQ(\D)$. However, this yields that $\Sigma(\D)$ would also be compact as a closed subspace of the compact space $HQ(\D)$, contradicting the non--compactness of $\Sigma(\D)$ (see \cite[Satz 1, p. 229]{Gaier}). Hence $HQ(\D)$ is not compact.
\end{proof}

An elementary persistence property in the interplay between harmonic and holomorphic mappings is that the post--composition of a holomorphic function with a harmonic one remains harmonic (see \cite[p. 2]{Duren}). This fact is utilized in order to prove

\begin{theorem} \label{TheoremDenseInItself} \hspace{0.1cm} \newline
The space $HQ(\D)$ is dense--in--itself, i.e. it does not contain any isolated points.
\end{theorem}
\begin{proof}
The space $Q(\D)$ is a topological group (see \cite[Theorem 3, p. 3]{BiersackLauf}) and not discrete (\cite[Theorem 12, p. 8]{BiersackLauf}); in particular, $\Sigma(\D)$ is not discrete (as already noticed in \cite[p. 230]{Gaier}). Hence, let $h \in HQ(\D)$ be arbitrary and choose a sequence $(f_n)_{n \in \N}$ in $\Sigma(\D) \backslash \{ \id_\D \}$ converging to $\id_\D$. Then, for each $n \in \N$, the mapping $g_n := h \circ f_n$ is harmonic and quasiconformal, thus $(g_n)_n$ is a sequence in $HQ(\D)$. The continuity of left multiplication in the topological group $Q(\D)$ yields $d_{\sup}(g_n, h) = d_{\sup}(h \circ f_n, h) \overset{n \to \infty}{\longrightarrow} 0$ due to $d_{\sup}(f_n, \id_\D) \overset{n \to \infty}{\longrightarrow} 0$.
\end{proof}

Combining Theorem \ref{TheoremClosedNotOpen} and Theorem \ref{TheoremDenseInItself} yields

\begin{corollary} \hspace{0.1cm} \newline
The space $HQ(\D)$ is perfect, i.e. it is closed in $Q(\D)$ and contains no isolated points.
\end{corollary}

Another property to be studied is the path--connectedness of $HQ(\D)$. In this context, the following integral operator will be of crucial importance (see \cite[p. 367]{Pavlovic} and \cite[p. 305]{PavlovicBook}):

\begin{definition}[Hilbert transformation] \hspace{0.1cm} \label{DefinitionHilbertTransformation} \newline
For periodic $\varphi \in L^1([0, 2 \pi])$ and $x \in \R$, the expression
\begin{align}
\HT(\varphi)(x) := - \frac{1}{\pi} \lim \limits_{\epsilon \to 0^+} \int \limits_{\epsilon}^{\pi} \frac{\varphi(x + t) - \varphi(x - t)}{2 \tan(t/2)} \, \operatorname{d}t
\label{FormulaHilbertTransformationDefinition}
\end{align}
is called the (periodic) Hilbert transformation of $\varphi$.
\end{definition}

\begin{remark} \label{RemarkHilbertTransformation} \hspace{0.1cm}
\begin{compactenum}[(i)]
	\item{In Fourier theory and trigonometric series, the Hilbert transformation plays a prominent role. However, the definition of the operator $\HT$ is not completely consistent in the vast literature about this topic. For example, a different formulation is given by
	\[
	\HT(\varphi)(x) = - \frac{1}{\pi} \lim \limits_{\epsilon \to 0^+} \int \limits_{\epsilon}^{\pi} \frac{\varphi(x + t) - \varphi(x - t)}{t} \, \operatorname{d}t,
	\]
	which is -- at least for existence questions -- equivalent to (\ref{FormulaHilbertTransformationDefinition}) due to $2 \tan(t/2) - t = 0$ for $t \longrightarrow 0$ (see \cite[p. 306]{PavlovicBook} and \cite[Vol. I, p. 52]{Zygmund}).}
	\item{The notion of Hilbert transformation is also present in further mathematical areas, for example in the classical theory of quasiconformal mappings in $\C$ (see \cite[pp. 156--160]{LehtoVirtanen}) and Teichm\"{u}ller spaces (see \cite[pp. 319--320]{GardinerLakic}). However, the circumstance that the definitions are in parts considerably different from each other is also present in these contexts.}
\end{compactenum}
\end{remark}

Due to the presence of the tangent function in the integrand's denominator in (\ref{FormulaHilbertTransformationDefinition}), the question for existence of $\mathfrak{H}$ raises, partially answered in (see \cite[p. 367]{Pavlovic} and \cite[Vol. I, p. 52]{Zygmund})

\begin{lemma} \hspace{0.1cm} \label{LemmaExistenceHilbertTransformationIntegral} \newline
%For $g \in L^1(\R)$, the Hilbert transform $H(g)$ as defined in (\ref{FormulaHilbertTransform}) exists almost everywhere.
For periodic $\varphi \in L^1([0, 2 \pi])$, the Hilbert transformation $\HT(\varphi)(x)$ exists for almost every $x \in \R$. Furthermore, $\HT(\varphi)(x)$ exists if $\varphi'(x)$ exists and is finite at $x \in \R$.
\end{lemma}

Now the connection between the Hilbert transformation $\HT$ and $HQ(\D)$ will be clarified. By the Rad\'{o}--Kneser--Choquet Theorem \ref{PropositionRKC}, every mapping $\gamma = e^{i \varphi} \in \mathcal{H}^+(\partial \D)$ defines a harmonic automorphism of $\D$ by means of the Poisson transformation $\PT[e^{i \varphi}]$ (this statement remains true even for $\gamma \in \mathcal{H}^\ast(\partial \D)$, see also \cite[(1.3), p. 338]{KrzyzNowak}). The question for whether this harmonic extension is quasiconformal has been answered in a characterizing manner by Pavlovi\'{c} in \cite{Pavlovic}, and is stated in (see \cite[Theorem 10.18, p. 305]{PavlovicBook})

\begin{proposition} \label{PropositionPavlovicAbstractHQC}\hspace{0.1cm} \newline
Let $f: \D \longrightarrow \D$ be an orientation--preserving harmonic homeomorphism of the unit disk onto itself. Then the following conditions are equivalent:
\begin{compactenum}[(i)]
	\item{$f$ is quasiconformal, i.e. $f \in HQ(\D)$;}
	%\item{$f$ is bi--Lipschitz with respect to the Euclidean metric;}
	\item{$f = \PT[e^{i \varphi}]$, where the function $\varphi$ has the following properties:
	\begin{enumerate}
		\item{$\varphi(t + 2\pi) - \varphi(t) = 2 \pi$ for all $t \in \R$;}
		\item{$\varphi$ is strictly increasing and bi--Lipschitz;}
		\item{the Hilbert transformation of $\varphi'$ is an element of $L^\infty(\R)$.}
	\end{enumerate}
	}
\end{compactenum}
\end{proposition}

A mapping $g: X \longrightarrow Y$ between metric spaces $(X, d_X)$ and $(Y, d_Y)$ is called \textit{bi--Lipschitz} if there exists a constant $L \in [1, + \infty)$ such that
\[
\frac{1}{L} d_X(x_1, x_2) \leq d_Y(g(x_1), g(x_2)) \leq L d_X(x_1, x_2)
\]
for all $x_1, x_2 \in X$, thus sharpening the classical notion of a Lipschitz--continuous mapping. In view of Corollary \ref{CorollaryCharacterizationHQD} and Proposition \ref{PropositionPavlovicAbstractHQC}, the following characterization for the elements of the space $HQ(\D)$ is valid: A harmonic (orientation--preserving) homeomorphism $\PT[e^{i \varphi}]$ of $\D$ onto itself is quasiconformal if and only if the corresponding mapping $\varphi$ is an element of
\begin{align*}
\HQC := \left\{ \varphi \in C([0,2\pi]) \; \bigg{|} \; \varphi \text{ is strictly increasing and bi--Lipschitz, } \varphi(2\pi) - \varphi(0) = 2\pi, \, \HT(\widetilde{\varphi}') \in L^\infty(\R) \right\}.
\end{align*}
Here, $\widetilde{\varphi}$ denotes the canonical extension of $\varphi \in \HQC$ to all of $\R$ via
\[
\widetilde{\varphi}(t + 2k\pi) := \varphi(t) + 2k\pi
\]
for all $t \in [0,2\pi]$ and every $k \in \Z$. By the requirement of strict increasing monotonicity, every mapping $\varphi \in \HQC$ is differentiable almost everywhere in (the interior of) $[0, 2 \pi]$. Consequently, each extended mapping $\widetilde{\varphi} \in C(\R)$ is differentiable almost everywhere in $\R$ with $\widetilde{\varphi}'$ being $2\pi$--periodic by construction. Furthermore, the assumption that $\varphi$ is bi--Lipschitz yields $\varphi' \in L^1([0, 2 \pi])$ (see \cite[Theorem 10, p. 124]{RoydenFitzpatrick}). Therefore, the condition $\HT(\widetilde{\varphi}') \in L^\infty(\R)$ is reasonable. In view of the path--connectedness of $HQ(\D)$, the first important observation to be made here is

\begin{lemma} \hspace{0.1cm} \label{LemmaSetIsConvex} \newline
The subset $\HQC \subsetneq C([0, 2 \pi])$ is convex. In particular, $\HQC$ is path--connected in the Banach space $C([0,2\pi])$.
\end{lemma}
\begin{proof}
Let $\varphi_1, \varphi_2 \in \HQC$, $\lambda \in [0,1]$ and consider the mapping $\lambda \varphi_1 + (1-\lambda) \varphi_2$.
\begin{compactenum}[(i)]
	\item{\uline{\textbf{Monotonicity:}} For $t, t' \in [0,2\pi]$ with $t < t'$, it is
	\[
	\lambda \varphi_1(t) + (1-\lambda) \varphi_2(t) < \lambda \varphi_1(t') + (1-\lambda) \varphi_2(t')
	\]
	due to $\lambda, (1-\lambda) \geq 0$, hence $\lambda \varphi_1 + (1-\lambda) \varphi_2$ is strictly increasing.}%It follows by definition of $\tilde{\varphi}$ that this extended mapping is increasing on $\R$ as well.}
	\item{\uline{\textbf{Bi--Lipschitz property:}} Let $t, t' \in [0,2\pi]$ and $L := \max\{L_1, L_2\}$ with $L_j$ denoting the bi--Lipschitz constant of $\varphi_j$, $j = 1,2$. Then on the one hand, by means of the triangle inequality, it is
	\begin{align*}
	\left| \lambda \varphi_1(t) + (1-\lambda) \varphi_2(t) - \lambda \varphi_1(t') - (1-\lambda) \varphi_2(t') \right| &\leq \lambda \left| \varphi_1(t) - \varphi_1(t') \right| + (1-\lambda) \left| \varphi_2(t) - \varphi_2(t') \right| \\
	&\leq \lambda L |t-t'| + (1 - \lambda) L |t-t'| = L |t-t'|.
	\end{align*}
	Hence $\lambda \varphi_1 + (1-\lambda) \varphi_2$ is Lipschitz--continuous with Lipschitz constant $L$. Without loss of generality, assume $t > t'$, then on the other hand, it is (recall that $\lambda \varphi_1 + (1-\lambda) \varphi_2$ is strictly increasing by (i))
	\begin{align*}
	\lambda \varphi_1(t) + (1-\lambda) \varphi_2(t) - \lambda \varphi_1(t') - (1-\lambda) \varphi_2(t') &= \lambda (\varphi_1(t) - \varphi_1(t')) + (1-\lambda) (\varphi_2(t) - \varphi_2(t') ) \\
	&\geq \lambda \frac{1}{L} (t-t') + (1 - \lambda) \frac{1}{L} (t-t') = \frac{1}{L} (t-t').
	\end{align*}
	Finally, switching the roles of $t$ and $t'$ shows that $\lambda \varphi_1 + (1-\lambda) \varphi_2$ is bi--Lipschitz continuous on $[0, 2\pi]$ with bi--Lipschitz constant $L$.}% An analogues reasoning shows that the extended mapping $\lambda \tilde{\varphi_1} + (1-\lambda) \tilde{\varphi_2}$ has the same property on $\R$.}
	\item{\uline{\textbf{Image interval has length $2 \pi$:}} It is
	\begin{align*}
	&\hspace{0.4cm} \lambda \varphi_1(2\pi) + (1-\lambda) \varphi_2(2\pi) - (\lambda \varphi_1(0) + (1-\lambda) \varphi_2(0)) \\
	&= \lambda (\varphi_1(2\pi) - \varphi_1(0)) + (1-\lambda) (\varphi_2(2\pi) - \varphi_2(0) ) \\
	&= \lambda 2\pi + (1-\lambda) 2\pi = 2 \pi.
	\end{align*}}
%	A simple calculation shows that the extended mapping $\lambda \tilde{\varphi_1} + (1-\lambda) \tilde{\varphi_2}$ fulfills
%	\[
%	\lambda \tilde{\varphi_1}(t+2k\pi) + (1-\lambda) \tilde{\varphi_2}(t+2k\pi) - \lambda \tilde{\varphi_1}(t) + (1-\lambda) \tilde{\varphi_2}(t) = 2k\pi
%	\]
%	for all $t \in [0, 2\pi]$ and every $k \in \Z$.}%Thus $\lambda \varphi_1 + (1-\lambda) \varphi_2$ is $2 \pi$--quasi periodic.}
	\item{\uline{\textbf{Hilbert transformation:}} First of all, $\lambda \widetilde{\varphi_1} + (1-\lambda) \widetilde{\varphi_2}$ is differentiable almost everywhere in $\R$ by (i) with
	\[
	\left( \lambda \widetilde{\varphi_1} + (1-\lambda) \widetilde{\varphi_2} \right)' = \lambda \widetilde{\varphi_1}' + (1-\lambda) \widetilde{\varphi_2}'.
	\]
	The function $\lambda \varphi_1' + (1-\lambda) \varphi_2'$ is contained in $L^1([0, 2 \pi])$ as the linear combination of such elements. Following Definition \ref{DefinitionHilbertTransformation}, the Hilbert transformation of $\lambda \widetilde{\varphi_1}' + (1-\lambda) \widetilde{\varphi_1}'$ is given by
	\begin{align*}
	\HT(\lambda \widetilde{\varphi_1}' + (1-\lambda) \widetilde{\varphi_2}')(x) &= - \frac{1}{\pi} \int \limits_{0^+}^\pi \frac{\lambda \widetilde{\varphi_1}'(x+t) + (1-\lambda) \widetilde{\varphi_2}'(x+t) - \lambda \widetilde{\varphi_1}'(x-t) - (1-\lambda) \widetilde{\varphi_2}'(x-t)}{2 \tan(t/2)} \, \mathrm{d}t \\
	&= - \frac{1}{\pi} \int \limits_{0^+}^\pi \frac{\lambda (\widetilde{\varphi_1}'(x+t) - \widetilde{\varphi_1}'(x-t)) + (1-\lambda) (\widetilde{\varphi_2}'(x+t) - \widetilde{\varphi_2}'(x-t))}{2 \tan(t/2)} \, \mathrm{d}t .
	\end{align*}
	Since $\varphi_1, \varphi_2 \in \HQC$, it is $\HT(\widetilde{\varphi_1}'), \HT(\widetilde{\varphi_2}') \in L^\infty(\R)$ by definition of the set $\HQC$, thus using the linearity of (improper) integrals the previous equation can be rewritten as
	\begin{align*}
	\HT(\lambda \widetilde{\varphi_1}' + (1-\lambda) \widetilde{\varphi_2}')(x) &= - \frac{\lambda}{\pi} \int \limits_{0^+}^\pi \frac{\widetilde{\varphi_1}'(x+t) - \widetilde{\varphi_2}'(x-t)}{2 \tan(t/2)} \, \mathrm{d}t  - \frac{1 - \lambda}{\pi} \int \limits_{0^+}^\pi \frac{\widetilde{\varphi_2}'(x+t) - \widetilde{\varphi_2}'(x-t)}{2 \tan(t/2)} \, \mathrm{d}t  \\
	&= \lambda \HT(\widetilde{\varphi_1}')(x) + (1-\lambda) \HT(\widetilde{\varphi_2}')(x).
	\end{align*}
	Since $L^\infty(\R)$ is a $\R$--vector space, the previous equation yields $\HT(\lambda \widetilde{\varphi_1}' + (1-\lambda) \widetilde{\varphi_2}') \in L^\infty(\R)$.}
\end{compactenum}
All in all, the mapping $\lambda \varphi_1 + (1-\lambda) \varphi_2$ is contained in $\HQC$ for every $\lambda$, hence $\HQC$ is convex. Thus, as a subset of the normed vector space $C([0, 2 \pi])$, $\HQC$ is also path--connected.
\end{proof}

Continuing the investigation, the set $\HQC$ now gives rise to consider the mapping
\begin{align}
\label{FormulaMappingFofConvexSet}
\Lambda: \HQC \longrightarrow HQ(\D), \; \varphi \longmapsto \left( \D \ni z \longmapsto \Lambda(\varphi)(z) := \PT[e^{i \varphi}](z) \right).
\end{align}
By Corollary \ref{CorollaryCharacterizationHQD} and Pavlovi\'{c}'s Proposition \ref{PropositionPavlovicAbstractHQC}, the mapping $\Lambda$ is surjective. Endowing the involved sets in (\ref{FormulaMappingFofConvexSet}) with the respective metric structures concludes in%; nevertheless, $HQ(\D)$ is no proper subset of $F(M)$ since each $f = P[e^{i \varphi}] \in F(M)$ is harmonic and quasiconformal in $\D$ by the results of Rado--Kneser--Choquet and Pavlovic
%(however, $\Lambda$ is not injective since $\varphi \in \HQC$ and $(\varphi + 2\pi) \in \HQC$ have the same image in $HQ(\D)$ under $\Lambda$)

\begin{theorem} \hspace{0.1cm} \label{TheoremMappingFisContinuousAndBijective} \newline
The mapping $\Lambda: (\HQC, \dsup) \longrightarrow (HQ(\D), \dsup)$ as defined in (\ref{FormulaMappingFofConvexSet}) is continuous and surjective.
\end{theorem}
\begin{proof}
The fact that $\Lambda$ is surjective was already mentioned above. Hence, let $(\varphi_n)_{n \in \N}$ converge in $\HQC$ to $\varphi \in \HQC$. The characterization of elements in $HQ(\D)$ stated in Proposition \ref{PropositionPavlovicAbstractHQC} implies that $(\Lambda(\varphi_n))_{n \in \N}$ is a sequence in $HQ(\D)$ and $\Lambda(\varphi) \in HQ(\D)$. In particular, $\Lambda(\varphi_n)$ and $\Lambda(\varphi)$ are harmonic quasiconformal automorphisms of $\D$, continuous on $\overline{\D}$ and coincide with $e^{i \varphi_n}$ and $e^{i \varphi}$ on $\partial \D$, respectively (see also \cite[p. 12]{Duren}). Therefore, since $\Lambda(\varphi_n) - \Lambda(\varphi)$ is harmonic as well, the Maximum Principle for harmonic mappings applies, concluding in
\begin{align*}
\sup \limits_{z \in \D} \left| \Lambda(\varphi_n)(z) - \Lambda(\varphi)(z) \right| &= \sup \limits_{z \in \partial \D} \left| \Lambda(\varphi_n)(z) - \Lambda(\varphi)(z) \right| \\
&= \sup \limits_{t \in [0,2\pi]} \left| e^{i \varphi_n(t)} - e^{i \varphi(t)} \right| \\
&\leq \sup \limits_{t \in [0,2\pi]} \left| \varphi_n(t) - \varphi(t) \right| = \dsup(\varphi_n, \varphi).
\end{align*}
In the estimate, the elementary inequality $|e^{ix} - e^{iy}| \leq |x-y|$ for $x,y \in \R$ was used. The last expression tends to zero for $n \to \infty$, proving the continuity of $\Lambda$.
\end{proof}

Finally, combining the statements of Lemma \ref{LemmaSetIsConvex} and Theorem \ref{TheoremMappingFisContinuousAndBijective} yields

\begin{theorem} \hspace{0.1cm} \newline
The space $HQ(\D)$ is path--connected.
\end{theorem}
%\begin{proof}
%The set $\HQC \subsetneq C([0,2\pi])$ is convex and thus path--connected. Furthermore, the mapping $\Lambda$ in (\ref{FormulaMappingFofConvexSet}) maps $\HQC$ continuously onto $HQ(\D)$, hence $HQ(\D)$ is path--connected as well. %Applying the well--known topological group isomorphism between $Q(\D)$ and $Q(G)$ shows that $HQ(G)$ is also (path--)connected.
%\end{proof}

\section{Incompleteness of $HQ(\D)$: Statement, auxiliary results and proof}

This section is concerned with the proof of the following statement:

\begin{theorem} \hspace{0.1cm} \label{TheoremIncompletenessHQG} \newline
The space $HQ(\D)$ is incomplete.
\end{theorem}

In order to prove this claim, some helpful results are collected in the following. The principal idea of the proof of Theorem \ref{TheoremIncompletenessHQG} is to construct a sequence of homeomorphic mappings of the interval $[0,1]$ onto itself converging uniformly to the \textbf{Cantor function} $\mathcal{C}: [0, 1] \longrightarrow [0,1]$; for basic information on this function, see \cite{DovgosheyEtAl} and \cite[Section 2.7, pp. 49--53]{RoydenFitzpatrick}. A result of Bo\v{z}in and Mateljevi\'{c} shows that, via the Poisson transformation, an appropriately modified variant of the mapping $\mathcal{C}$ induces a harmonic homeomorphism of the unit disk $\D$ onto itself which is \textit{not} quasiconformal (see Proposition \ref{PropositionCounterexampleHQC}). However, this harmonic homeomorphism will be seen to arise as the uniform limit of harmonic quasiconformal automorphisms of $\D$, thus implying that $HQ(\D)$ cannot be complete. \newline

First of all, an approximation procedure for the Cantor function $\mathcal{C}$ in terms of a certain recursively defined sequence, which will be of central importance, is stated (see \cite[Proposition 4.2, p. 9]{DovgosheyEtAl}):

\begin{lemma} \hspace{0.1cm} \label{LemmaCantorFunctionApprox} \newline
Let $B([0,1])$ denote the Banach space of bounded real--valued functions on $[0,1]$. The Cantor function $\mathcal{C}$ is the unique element of $B([0,1])$ for which
\[
\mathcal{C}(x) = \begin{cases} \frac{1}{2} \mathcal{C}(3x), & 0 \leq x \leq \frac{1}{3} \\ \frac{1}{2}, & \frac{1}{3} < x < \frac{2}{3} \\ \frac{1}{2} + \frac{1}{2} \mathcal{C}(3x-2), & \frac{2}{3} \leq x \leq 1. \end{cases}
\]
Moreover, for arbitrary $\psi_0 \in B([0,1])$, the sequence $(\psi_n)_{n \in \N_0}$ defined by
\begin{align}
\label{FormulaRecursiveFunctionSequenceDefinition}
\psi_{n+1}(x) := \begin{cases} \frac{1}{2} \psi_n(3x), & 0 \leq x \leq \frac{1}{3} \\ \frac{1}{2}, & \frac{1}{3} < x < \frac{2}{3} \\ \frac{1}{2} + \frac{1}{2} \psi_n(3x-2), & \frac{2}{3} \leq x \leq 1 \end{cases}
\end{align}
for $n \in \N_0$ converges uniformly on $[0,1]$ to $\mathcal{C}$.
\end{lemma}

An approximation of the Cantor function using the recursively defined sequence given by (\ref{FormulaRecursiveFunctionSequenceDefinition}) is shown in Figure \ref{FigureApproximationCantorFunction}. Basically, the principal idea of the approximation procedure and the mappings $\psi_n$ is that the initial mapping $\psi_0$ is \enquote{copied} and gets \enquote{duplicated in a scaled fashion}, being added to the graph of $\psi_n$ more and more times as the index increases. This is visualized by the right--hand picture in Figure \ref{FigureApproximationCantorFunction}: In the first step (in blue), the scaled initial mapping $\psi_0$ can be seen two times, namely on the intervals $[0, \frac{1}{3}]$ and $[\frac{2}{3}, 1]$. After the second iteration (in orange), the mapping $\psi_0$ appears four times in a scaled manner. Finally, in the third step (in yellow), the appropriately scaled version of $\psi_0$ is present eight times. In particular, it becomes obvious that all continuity and differentiability questions regarding $\psi_n$ depend solely on the behaviour of the initial mapping $\psi_0$ (and eventually existing derivatives) at the boundary points $x = 0$ and $x = 1$ of the starting interval.

\begin{figure}
\centering
\includegraphics[width=16cm]{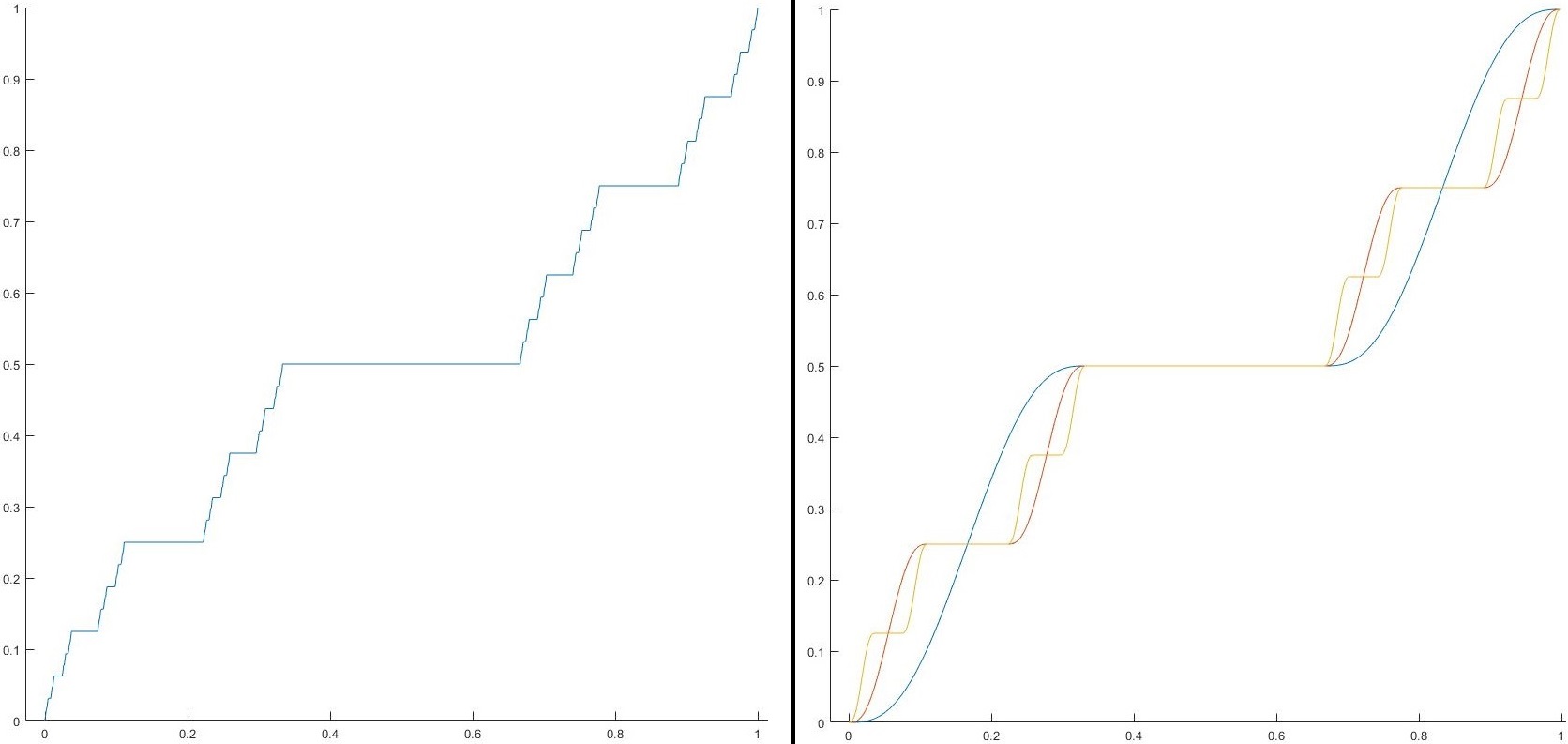}
\caption[An approximation of the Cantor function $\mathcal{C}$ together with the first three approximating functions.]{Left: An approximation of the Cantor function $\mathcal{C}$ using the recursively defined sequence $(\psi_n)_n$ described in Lemma \ref{LemmaCantorFunctionApprox}. The initial function is given by $\psi_0(x) = 6x^5 - 15x^4 + 10x^3$ for $x \in [0,1]$, and for the approximation, the index value $n = 15$ was chosen. Right: The first three function in the approximation sequence $(\psi_n)_n$: $\psi_1$ in blue, $\psi_2$ in orange and $\psi_3$ in yellow.}
\label{FigureApproximationCantorFunction}
\end{figure}

Furthermore, in Lemma \ref{LemmaCantorFunctionApprox}, the stated approximation part and the related uniqueness of $\mathcal{C}$ is based on Banach's Contraction Principle (see \cite[p. 216]{RoydenFitzpatrick}). The following Proposition contains the mentioned result of Bo\v{z}in/Mateljevi\'{c} concerning a harmonic homeomorphism of $\D$ which fails to be quasiconformal (see \cite[Example 3.2, pp. 29--30]{BozinMateljevic}):

\begin{proposition} \hspace{0.1cm} \label{PropositionCounterexampleHQC} \newline
For $t \in [0,2 \pi]$, define $\varphi_{\mathcal{C}}(t) := \pi (\mathcal{C}(\frac{t}{2 \pi}) + \frac{t}{2 \pi})$ and $\gamma_{\mathcal{C}}(t) := e^{i \varphi_{\mathcal{C}}(t)}$. Then the function $h_{\mathcal{C}} := \PT[\gamma_{\mathcal{C}}]$ is a harmonic homeomorphism of $\D$ onto itself that is not quasiconformal.
\end{proposition}

Now all preparations are made in order to prove the claim of Theorem \ref{TheoremIncompletenessHQG}:

\begin{proof}[Proof of Theorem \ref{TheoremIncompletenessHQG}]
Consider the polynomial function
\[
\psi_0: [0,1] \longrightarrow \R, \, x \longmapsto \psi_0(x) := 6x^5-15x^4+10x^3
\]
whose first and second derivatives satisfy
\begin{align}
\psi_0'(0) = \psi_0'(1) = 0 = \psi_0''(0) = \psi_0''(1).
\label{FormulaValuesDerivativesPsi}
\end{align}
Furthermore, $\psi_0$ is strictly increasing on $(0,1)$ and leaves the boundary points fixed -- in other words, $\psi_0$ maps $[0,1]$ homeomorphically onto itself. Lemma \ref{LemmaCantorFunctionApprox} implies that the corresponding sequence $(\psi_n)_{n \in \N_0}$ defined via (\ref{FormulaRecursiveFunctionSequenceDefinition}) converges uniformly on $[0,1]$ to the Cantor function $\mathcal{C}$, and by construction, it is $\psi_n \in C^2([0,1])$ for every $n \in \N_0$ due to (\ref{FormulaValuesDerivativesPsi}). Transferring the $\psi_n$ to the interval $[0, 2\pi]$ via
\begin{align}
\varphi_n(t) := \pi \left( \psi_n \left( \frac{t}{2 \pi} \right) + \frac{t}{2 \pi} \right), \; t \in [0, 2 \pi],% \; \text{ and } \; \varphi(t) := \pi \left( \mathcal{C} \left( \frac{t}{2 \pi} \right) + \frac{t}{2 \pi} \right)
\label{FormulaDefinitionVarphiMappings}
\end{align}
yields a sequence $(\varphi_n)_n$ of $C^2$--homeomorphism of $[0,2 \pi]$ onto itself. Accordingly, this sequence $(\varphi_n)_n$ clearly converges uniformly on $[0,2\pi]$ to the mapping $\varphi_{\mathcal{C}}$ defined in Proposition \ref{PropositionCounterexampleHQC}. As a next step, the mappings $\varphi_n$ and $\varphi_{\mathcal{C}}$ are extended to all of $\R$ by setting
\begin{align}
\varphi_n(t + 2 k \pi) := \varphi_n(t) + 2 k \pi% \; \text{ and } \; \varphi(t + 2 k \pi) := \varphi(t) + 2 k \pi
\label{FormulaDefinitionVarphiMappingsExtension}
\end{align}
for $k \in \Z$ and $t \in [0, 2 \pi]$, yielding a sequence $(\varphi_n)_{n \in \N_0} \subseteq C^2(\R)$; likewise, the mappings $\psi_n$ and $\mathcal{C}$ are extended in the same manner (the extended mappings are denoted by the same letter). In particular, the $\varphi_n$ are differentiable with $\varphi'_n(t + 2 k \pi) = \varphi'_n(t)$ for all $t \in \R$ by construction, i.e. the $\varphi'_n$ (and thus the $\varphi''_n$ as well) are continuous $2 \pi$--periodic mappings. Lifting these mappings to the unit circle by
\[
\gamma_n(e^{it}) := e^{i \varphi_n(t)}% \; \text{ and } \; f(e^{it}) := e^{i \varphi(t)} 
\]
for $t \in [0, 2 \pi]$ and each $n \in \N$ yields orientation--preserving homeomorphisms of $\partial \D$ onto itself, hence the harmonic extensions $\PT[\gamma_n]$ by means of the Rad\'{o}--Kneser--Choquet Theorem \ref{PropositionRKC} are (orientation--preserving) harmonic homeomorphisms of $\D$ onto itself. In order to visualize this procedure and to illustrate a concrete mapping of the described type, the mapping behaviour of the harmonic unit disk automorphism $\PT[\gamma_4]$ is visualized in Figure \ref{FigureHarmonicQCAutomorphismUnitDisk}.

\begin{figure}
\centering
\includegraphics[width=16cm]{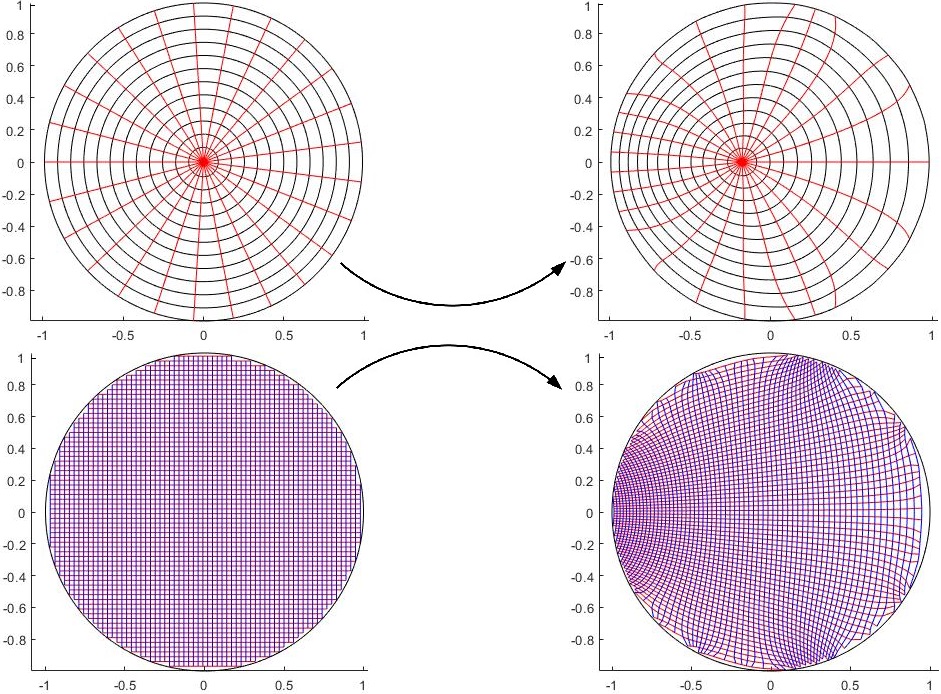}
\caption[An approximation of the harmonic quasiconformal unit disk automorphism $\PT{[\gamma_4]}$.]{Left--hand side: Preimage of concentric circles and radial rays (top) and of an Euclidean grid (bottom) in the unit disk. Right--hand side: Image of the concentric circles, radial rays and the Euclidean grid in $\D$ under the mapping $\PT[\gamma_4]$.}
\label{FigureHarmonicQCAutomorphismUnitDisk}
\end{figure}

Moreover, by Pavlovi\'{c}'s characterization result stated in Proposition \ref{PropositionPavlovicAbstractHQC}, the mappings $\PT[\gamma_n]$ in fact define quasiconformal automorphisms of $\D$, which can be seen as follows: \newline
%By Proposition \ref{PropositionCounterexampleHQC}, the mapping $h$ is not quasiconformal, i.e. $P[f] \not \in HQ(\D)$. However,

It is $\varphi_n \in C^2(\R)$ strictly increasing with $\varphi_n(t + 2 \pi) = \varphi_n(t) + 2 \pi$ for all $t \in \R$ by construction, see (\ref{FormulaDefinitionVarphiMappings}) and (\ref{FormulaDefinitionVarphiMappingsExtension}). Furthermore, as $C^2$--homeomorphisms, each mapping $\varphi_n$ is Lipschitz--continuous, and the corresponding inverse mappings $\varphi_n^{-1}$ are also $C^2$ by construction due to (\ref{FormulaValuesDerivativesPsi}), thus also Lipschitz--continuous. In consequence, the mappings $\varphi_n$ are bi--Lipschitz. Hence, in view of Proposition \ref{PropositionPavlovicAbstractHQC}(ii), the Hilbert transformation condition (c) needs to be verified. Therefore, let $x \in \R$, then it is
\[
\left| \varphi'_n(x + t) - \varphi'_n(x - t) \right| \leq L_n \cdot | x + t - (x - t)| = 2 L_n |t|,
\]
since $\varphi_n \in C^2(\R)$, thus $\varphi'_n$ is Lipschitz--continuous on $\R$ with Lipschitz constant $L_n \in \R^+$. This yields
\begin{align*}
\left| \int \limits_{0^+}^\pi \frac{\varphi_n'(x+t) - \varphi_n'(x-t)}{t} \mathrm{d}t \right| &\leq \int \limits_{0^+}^\pi \frac{\left| \varphi_n'(x+t) - \varphi_n'(x-t) \right|}{t} \mathrm{d}t \leq \int \limits_{0^+}^\pi \frac{2 L_n t}{t} \mathrm{d}t = 2 \pi L_n < + \infty,
\end{align*}
and now Remark \ref{RemarkHilbertTransformation}(i) implies that $\HT(\varphi'_n)$ is (essentially) bounded for $\varphi_n, n \in \N_0$ (note that the conclusion could also have been drawn from Lemma \ref{LemmaExistenceHilbertTransformationIntegral} since $\varphi'_n$ and $\varphi''_n$ are periodic and continuous on $\R$). Thus Proposition \ref{PropositionPavlovicAbstractHQC} shows that the mappings $\PT[\gamma_n]$ are quasiconformal automorphisms of $\D$. \newline

Finally, it will be shown that the mappings $\PT[\gamma_n]$ converge uniformly on $\D$ to the non--quasiconformal mapping $h_{\mathcal{C}}$ in question (from Proposition \ref{PropositionCounterexampleHQC}), which is essentially based on the same idea as the proof of Theorem \ref{TheoremMappingFisContinuousAndBijective}: Applying the Maximum Principle for harmonic functions to $\PT[\gamma_n] - h_{\mathcal{C}}$ yields
%(note that the Rado--Kneser--Choquet Theorem \ref{PropositionRKC} assures that these expressions are well--defined on $\overline{\D}$ and $\partial \D$, respectively)
% For $z = r e^{i \phi} \in \D$, it is
%\begin{align*}
%\left| P[f_n](z) - P[f](z) \right| &= \frac{1}{2 \pi} \left| \int \limits_0^{2 \pi} \frac{1 - r^2}{1 - 2r \cos(t - \phi) + r^2} (e^{i \varphi_n(t)} - e^{i \varphi(t)}) \, \mathrm{d}t \right| \\
%&\leq \frac{1}{2 \pi} \int \limits_0^{2 \pi} \frac{1 - r^2}{1 - 2r \cos(t - \phi) + r^2} \left| e^{i \varphi_n(t)} - e^{i \varphi(t)} \right| \, \mathrm{d}t \\
%&\leq \frac{1}{2 \pi} \int \limits_0^{2 \pi} \frac{1 - r^2}{1 - 2r \cos(t - \phi) + r^2} \left| \varphi_n(t) - \varphi(t) \right| \, \mathrm{d}t \\
%&\leq \frac{\left\| \varphi_n - \varphi \right\|_{\sup} }{2 \pi} \int \limits_0^{2 \pi} \frac{1 - r^2}{1 - 2r \cos(t - \phi) + r^2} \, \mathrm{d}t
%\end{align*}
\begin{align*}
\sup \limits_{z \in \D} \left| \PT[\gamma_n](z) - h_{\mathcal{C}}(z) \right| &= \max \limits_{z \in \partial \D} \left| \PT[\gamma_n](z) - \PT[\gamma_{\mathcal{C}}](z) \right| = \max \limits_{t \in [0, 2\pi]} \left| e^{i \varphi_n(t)} - e^{i \varphi_{\mathcal{C}}(t)} \right| \leq \max \limits_{t \in [0, 2\pi]} \left| \varphi_n(t) - \varphi_{\mathcal{C}}(t) \right|.
\end{align*}
%In the last step, the inequality $|e^{ix} - e^{iy}| \leq |x-y|$ for $x,y \in \R$ was used.
Since $\varphi_n$ converges uniformly on $[0, 2 \pi]$ to $\varphi$, the claim follows: The sequence $(\PT[\gamma_n])_{n \in \N_0}$ in $HQ(\D)$ converges uniformly to $h_{\mathcal{C}} \not \in HQ(\D)$, showing that the space $HQ(\D)$ is incomplete.
%The standard argument using the topological group isomorphism between $Q(\D)$ and $Q(G)$ for bounded Jordan domains $G \subsetneq \C$ (as carried out in several other working papers of the author) shows that $HQ(G)$ is $d_{\sup}$--incomplete as well.
%(this, in turn, follows from the fact that $\psi_n$ converges uniformly on $[0,1]$ to the Cantor function $\mathcal{C}$)
\end{proof}

%In particular, by considering the homeomorphically extended mappings on $\overline{\D}$, the proof of Theorem \ref{TheoremIncompletenessHQG} immediately implies the following statement, which will be of central importance in Section \ref{SectionIncompletenessRevisited}:
%
%\begin{corollary} \hspace{0.1cm} \label{CorollaryHQDnotClosedInQD} \newline
%The subspace $Q(\overline{\D})$ is not closed in the homeomorphism group $\mathcal{H}(\overline{\D})$.
%\end{corollary}
%\begin{proof}
%The homeomorphic extensions to $\overline{\D}$ of the mappings $\PT[\gamma_n], n \in \N$, in $HQ(\D) \subsetneq Q(\D)$ converge uniformly on $\overline{\D}$ to (the extension of) $h_{\mathcal{C}} \in \mathcal{H}(\overline{\D}) \backslash Q(\overline{\D})$, and the claim follows.
%\end{proof}


\begin{thebibliography}{9}

\bibitem{AxlerEtAl}
AXLER, S., BOURDON, P., RAMEY, W.: \textit{Harmonic Function Theory}, 2. edition. Graduate Texts in Mathematics, vol.~137. Springer--Verlag New York, 2001.

\bibitem{BiersackLauf}
BIERSACK, F., LAUF, W.: \textit{Topological Properties of Quasiconformal Automorphism Groups}, The Journal of Analysis, to appear.

\bibitem{BozinMateljevic}
BO\v{Z}IN, V., MATELJEVI\'{C}, M.: \textit{Some counterexamples related to the theory of HQC mappings}. Filomat, \textbf{24}(4), pp. 25--34 (2010).

\bibitem{BshoutyHengartner}
BSHOUTY, D., HENGARTNER, W.: \textit{Univalent Harmonic Mappings in the Plane}. Annales Universitatis Mariae Curie--Sklodowska, Sectio A -- Mathematica, \textbf{48}(3), pp. 12--42 (1994).

\bibitem{DovgosheyEtAl}
DOVGOSHEY, O., MARTIO, O., RYAZANOV, V., VUORINEN, M.: \textit{The Cantor function}. Expositiones Mathematicae, \textbf{24}(1), pp. 1--37 (2006).

\bibitem{Duren}
DUREN, P.: \textit{Harmonic Mappings in the Plane}. Cambridge Tracts in Mathematics, no.~156. Cambridge University Press, 2004.

\bibitem{DurenSchober}
DUREN, P., SCHOBER, G.: \textit{A Variational Method for Harmonic Mappings onto Convex Regions}. Complex Variables, \textbf{9}(2--3), pp. 153--168 (1987).

\bibitem{Gaier}
GAIER, D.: \textit{\"{U}ber R\"{a}ume konformer Selbstabbildungen ebener Gebiete}. Mathematische Zeitschrift, \textbf{187}(2), pp. 227--257 (1984).

\bibitem{GardinerLakic}
GARDINER, F.P., LAKIC, N.: \textit{Quasiconformal Teichm\"{u}ller Theory}. Mathematical Surveys and Monographs, vol.~76. American Mathematical Society, 2000.

\bibitem{Kalaj}
KALAJ, D.: \textit{Quasiconformal and harmonic mappings between Jordan domains}. Mathematische Zeitschrift, \textbf{260}(2), pp. 237--252 (2008).

\bibitem{KrzyzNowak}
KRZYZ, J.G., NOWAK, M.: \textit{Harmonic automorphisms of the unit disk}. Journal of Computational and Applied Mathematics, \textbf{105}(1--2), pp. 337--346 (1999).

\bibitem{Lehto}
LEHTO, O.: \textit{Univalent Functions and Teichm\"{u}ller Spaces}. Graduate Texts in Math\nolinebreak ematics, vol.~109. Springer--Verlag New York, 1987.

\bibitem{LehtoVirtanen}
LEHTO, O., VIRTANEN, K.I.: \textit{Quasiconformal Mappings in the Plane}, 2. edition. Die Grundlehren der mathematischen Wissenschaften, Band~126. Springer--Verlag Berlin Heidelberg New York, 1973.

\bibitem{Pavlovic}
PAVLOVI\'{C}, M.: \textit{Boundary Correspondence under Harmonic Quasiconformal Homeomorphisms of the Unit Disk}. Annales Academiae Scientiarum Fennicae. Mathematica. vol.~27, pp. 365--372 (2002).

\bibitem{PavlovicBook}
PAVLOVI\'{C}, M.: \textit{Function Classes on the Unit Disk: An Introduction}. De Gruyter Studies in Mathematics, vol.~52. Walter de Gruyter, Berlin/Boston, 2014.

\bibitem{RoydenFitzpatrick}
ROYDEN, H.L., FITZPATRICK, P.M.: \textit{Real Analysis}, 4. edition. Prentice Hall, 2010.

\bibitem{Zygmund}
ZYGMUND, A.: \textit{Trigonometric Series}, 3. edition, Volumes I \& II combined, with a foreword by R. Fefferman. Cambridge Mathematical Librar Series. Cambridge University Press, 2002.
\end{thebibliography}
\end{document}